\documentclass{amsart}

\usepackage{amsthm, amsfonts, amssymb, amsmath, latexsym, enumerate, times}
\usepackage[latin1]{inputenc}
\usepackage{mathrsfs}
\usepackage{mathtools}
\usepackage{letltxmacro}
\LetLtxMacro\orgvdots\vdots
\LetLtxMacro\orgddots\ddots

\makeatletter
\DeclareRobustCommand\vdots{%
	\mathpalette\@vdots{}%
}
\newcommand*{\@vdots}[2]{%
	\sbox0{$#1\cdotp\cdotp\cdotp\m@th$}%
	\sbox2{$#1.\m@th$}%
	\vbox{%
		\dimen@=\wd0 %
		\advance\dimen@ -3\ht2 %
		\kern.5\dimen@
		\dimen@=\wd2 %
		\advance\dimen@ -\ht2 %
		\dimen2=\wd0 %
		\advance\dimen2 -\dimen@
		\vbox to \dimen2{%
			\offinterlineskip
			\copy2 \vfill\copy2 \vfill\copy2 %
		}%
	}%
}
\DeclareRobustCommand\ddots{%
	\mathinner{%
		\mathpalette\@ddots{}%
		\mkern\thinmuskip
	}%
}
\newcommand*{\@ddots}[2]{%
	\sbox0{$#1\cdotp\cdotp\cdotp\m@th$}%
	\sbox2{$#1.\m@th$}%
	\vbox{%
		\dimen@=\wd0 %
		\advance\dimen@ -3\ht2 %
		\kern.5\dimen@
		\dimen@=\wd2 %
		\advance\dimen@ -\ht2 %
		\dimen2=\wd0 %
		\advance\dimen2 -\dimen@
		\vbox to \dimen2{%
			\offinterlineskip
			\hbox{$#1\mathpunct{.}\m@th$}%
			\vfill
			\hbox{$#1\mathpunct{\kern\wd2}\mathpunct{.}\m@th$}%
			\vfill
			\hbox{$#1\mathpunct{\kern\wd2}\mathpunct{\kern\wd2}\mathpunct{.}\m@th$}%
		}%
	}%
}
\makeatother

\usepackage{array}
\usepackage{comment}
\usepackage{paralist}
\usepackage{xcolor}

\newtheorem{theorem}{Theorem}

\newtheorem{corollary}[theorem]{Corollary}
\newtheorem{proposition}[theorem]{Proposition}

\theoremstyle{definition}
\newtheorem{definition}[theorem]{Definition}
\newtheorem{remark}[theorem]{Remark}

\newcommand{\calC}{{\mathcal C}}
\newcommand{\calD}{{\mathcal D}}

\newcommand{\calF}{{\mathcal F}}
\newcommand{\calG}{{\mathcal G}}
\newcommand{\calI}{{\mathcal I}}

\newcommand{\bbP}{{\mathbb P}}

\def\geq{\geqslant}
\def\leq{\leqslant}

\begin{document}
 
\title[Variations on Pascal's Theorem]{Variations on Pascal's Theorem}

\author{Ciro Ciliberto}
\address{Dipartimento di Matematica, Universit\`a di Roma Tor Vergata, Via O. Raimondo
 00173 Roma, Italia}
\email{cilibert@axp.mat.uniroma2.it}

\author{Rick Miranda}
\address{Department of Mathematics, Colorado State University, Fort Collins (CO), 80523,USA}
\email{rick.miranda@colostate.edu}
 
\subjclass{Primary 14H45, 14N05, 14N15, 14N20, 14N25; Secondary 14A25, 14Q15}
 
\keywords{rational normal curves, projective geometry, Pascal's theorem}
 
\maketitle

\tableofcontents

\begin{abstract}
In this paper we present a variety of statements that are in the spirit of the famous theorem of Pascal, often referred to as the Mystic Hexagon.  We give explicit equations describing the conditions for $d+4$ points to lie on rational normal curves.  A collection of problems of Pascal type are considered for quadric surfaces in $\bbP^3$.  Finally we reprove, using computer algebra methods, a remarkable theorem of Richmond, Segre, and Brown, for quadrics in $\bbP^4$ containing five general lines. 
\end{abstract}

\section*{Introduction} 

We recall the famous theorem of Pascal,
which states that for six points $p_i$ ($1\leq i \leq 6$) 
in general position in the plane,
they lie on a conic 
if and only if
the three derived points of intersection of the opposite lines of the hexagon 
with vertices equal to the $p_i$ lie on a line. 
We give the precise statement in Theorem \ref{PascalTheorem}.

Various authors have considered generalizations of Pascal's theorem
to prove statements which we will call ``of Pascal type".
The general form of such statements is the following.
Let $\calF$ and $\calG$ be irreducible families of varieties in a projective space.
Let $\calC$ be a collection of subsets of this projective space,
from which one can form a derived collection of subsets $\calD$ (also in this space)
using elementary projective geometry constructions
(e.g., taking spans, intersections, etc.).

A statement of Pascal type would be one with the following form:
\emph{There is a member of $\calF$ containing the collection $\calC$
if and only if there is a member of $\calG$ containing the derived collection $\calD$.}

For example Pascal's original theorem 
has $\calF$ being the space of conics in a given plane,
$\calG$ the space of lines, $\calC$ a sextuple of points, and $\calD$ triples of points,
with the triples being defined by the sextuples 
in the way indicated in Pascal's Theorem.

In this paper we will consider some instances of theorems of Pascal type.
When $\calF$ is the family of rational normal curves of degree $d$ in $\bbP^d$,
Caminata and Schaffler in \cite[Theorem A]{CS} have generalized  Pascal's Theorem   
when $\calC$ is the space of ($d+4$)-tuples of points 
in general position in $\bbP^d$.

When $\calF$ is the family of quadric hypersurfaces in $\bbP^4$,
and $\calC$ is a collection of five lines in general position there,
a remarkable theorem of Pascal type 
has been observed by Richmond \cite{R} in 1899
(in one direction) and full proofs were given 
by B. Segre \cite{S} and Brown \cite{B}
in 1945.

In this paper we first present a computer algebra approach to the original Pascal's Theorem,
first executed (it seems) by Stefanovi'c  and Milo'sevi'c in \cite{SM}.
Then we consider the problem of characterizing points on rational normal curves,
and give a condition for $d+4$ points to lie on one,
which simplifies the Caminata-Schaffler statement.

Secondly we turn our attention to the case 
of ten points on a quadric surface in $\bbP^3$;
we have partial results which are adjacent to but not exactly of Pascal type.

Finally we take up the case of five lines on quadrics in $\bbP^4$,
and re-prove the Richmond-Segre-Brown result via computer algebra methods,
which is in the same spirit as the approach in \cite{SM}.
The proof follows Segre's approach, but is purely algebraic,
and avoids the long complicated geometric arguments that Segre and Brown use. 

\section{Points on rational normal curves}
Any set of $d+3$ points in general position on $\bbP^d$
lie on a unique rational normal curve.

One can ask: given $d+4$ points in general position in $\bbP^d$,
how can we detect whether they lie on a rational normal curve?

For $d=2$, the question is answered by Pascal's Theorem for conics in the plane;
we remind the reader how this goes in Subsection \ref{n=2section}.

The case $n=3$ has been considered by Chasles \cite{Chasles};
he proposes a rather complicated condition which we do not see how to extend to higher dimensions.

We present an iterative approach that works for all $n$, and builds on Pascal's Theorem for $n=2$.  This is illustrated for $d=3$ in subsection \ref{RNCn=3}
before turning our attention to the general case in subsection \ref{RNCgeneral}.

\subsection{The case $d=2$}\label{n=2section}
We recall the famous ``Mystic Hexagon" Theorem of Pascal:

\begin{theorem}\label{PascalTheorem}
Let $p_1,\ldots,p_6$ be six points in the plane, no three collinear.
Let $L_i$ be the line joining $p_i$ to $p_{i+1}$ (indices taken modulo $6$),
for $1 \leq i \leq 6$.
Let $q_j$ be the point $L_j \cap L_{j+3}$ for $1 \leq i \leq 3$.
Then
the six points $\{p_i\}_{1 \leq i \leq 6}$ lie on an irreducible conic
if and only if the three points $\{q_j\}_{1 \leq j \leq 3}$ are collinear.
\end{theorem}

There are a number of proofs of this theorem; it may be that  one of  the most recent and modern proof (using symbolic algebra packages) was given in \cite{SM} and goes as follows.

Since no three of the points are collinear, the $p_i$'s lie on an irreducible conic
if and only if they lie on a conic, i.e., they satisfy a quadratic polynomial in the coordinates of $\bbP^2$.
If we choose such coordinates $[x_i:y_i:z_i]$ for $p_i$,
then a quadratic polynomial
$ax^2+by^2+cz^2+ dxy+exz+fyz$ has $p_i$ as a root if and only if
the row vector $v_i = (x_i^2,y_i^2,z_i^2,x_iy_i,x_iz_i,y_iz_i)$
is orthogonal to the column vector $(a,b,c,d,e,f)^\top$.
Hence to find such a quadratic polynomial defining the conic,
we must find such a nonzero column vector
that is orthogonal to the six row vectors $v_i$ given by each of the points.
This is possible if and only if the determinant
of the $6\times 6$ matrix whose rows are the $v_i$'s is zero.
This determinant is a polynomial $F$ in the $18$ variables which 
define the coordinates of the points.
It is quadratic in each of the sets of six variables for each point,
and has therefore total degree $12$.

On the other hand, each point $q_j$ has projective coordinates
given by quadratic polynomials in the six variables defining the two points $p_j$ and $p_{j+3}$, given by the $2\times 2$ minors of the corresponding $2\times 3$ matrix formed by the coordinates of $p_j$ and $p_{j+3}$.
These expressions for the coordinates of the $q_j$ are quartic polynomials in the $p$ coordinates.
Then the three $q_j$'s are collinear if and only if the determinant of the $3\times 3$ matrix whose rows are their coordinates is zero.
This is a polynomial $G$, again in the $18$ variables.
It is cubic in each of the three sets of $q$ variables,
hence of total degree $12$ in the $p$ variables also.

Any modern symbolic algebra package (we verified this with SageMath) 
can deduce that
$F=G$ as polynomials in the $18$ variables, 
which was first noted in \cite{SM}.

We observe that the identity $F=G$ proves a bit more, namely Pappus' Theorem,
which deals with the case when the points may lie three each on two lines.

\subsection{The case $n=3$}\label{RNCn=3}
Our analysis of the general case is illustrated by the case $n=3$,
which we now take up as an example.
Let $p_1,\ldots,p_7$ be seven points in general position in $\bbP^3$.
We note that if we project from any one of the points, say $p_i$,
the remaining six go to a set $S_i$ of six points in $\bbP^2$,
which are also in general position.

\begin{proposition}
The points $\{p_i\}_{1\leq i\leq 7}$ lie on a twisted cubic curve in $\bbP^3$
if and only if two of the sets $S_i$ and $S_j$,  with $1\leq i<j\leq 7$,  
lie on a conic in $\bbP^2$.
\end{proposition}

\begin{proof}
If the points lie on a twisted cubic, then it is clear that all of the sets $S_i$
lie on the image of the cubic, which is a conic.

Conversely, suppose that $S_i$ and $S_j$ lie conics $C_i$ and $C_j$ respectively.
Consider the two cones $T_i$ and $T_j$ formed by the two joins of the conics and the corresponding projection points $p_i$ and $p_j$.
These are quadric surfaces that meet in the line $L$ joining $p_i$ to $p_j$.
The residual intersection is the desired twisted cubic. \end{proof}

To implement this criterion with $7$ specific points,
we make two explicit projections,
and then apply the two conditions of Pascal's Theorem,
which involve checking two determinants being zero.
This  of course gives two conditions on the coordinates of the seven points.

This is the correct number: the first six points lie on a unique twisted cubic,
and it is two conditions for the seventh point in $\bbP^3$ to lie on this curve.

We can make another count of parameters as follows.
The space of twisted cubics depend on $12$ parameters;
such a curve is parametrized by four cubic polynomials, which have then $16$ coefficients;
the automorphisms of the parametrization are the three dimensions of automorphisms of $\bbP^1$ and the scaling in $\bbP^3$.
Alternatively, they are all projectively equivalent in $\bbP^3$,
and the $15$ dimensions of automorphisms of $\bbP^3$ give the $12$ parameters,
since there is a $3$-dimensional stabilizer for any one of them.

Then seven points on these cubics depend on $12+7=19$ parameters.
But seven points in $\bbP^3$ depend on $21$ parameters, and $21-91=2$.

\subsection{The general case}\label{RNCgeneral}
Now we fix $d+4$ points $\{p_i\}_{1\leq i\leq d+4}$ in general position in $\bbP^d$.
Note first that if we choose a point $p_i$,
we may project the others to $\bbP^{d-1}$,
obtaining a set $S_i$ of $d+3$ points in general position in $\bbP^{d-1}$.

Note secondly that if we choose a set $R$ of $d-2$ of the points,
where $p_i \notin R$,
we may project (from $R$) the remaining six points to a set $T_R$ in $\bbP^2$.

\begin{theorem}
The set of $d+4$ points in $\bbP^d$ lie on a rational normal curve of degree $d$
if and only if
$S_i$ lies on a rational normal curve of degree $d-1$ in $\bbP^{d-1}$
and $T_R$ lies on a conic in $\bbP^2$.
\end{theorem}

\begin{proof} 
If the $d+4$ points lie on a rational normal curve, the assertion is true in that direction.
Conversely, we have two cones at hand:
one is the cone $C_i$ over the rational normal curve in $\bbP^{d-1}$ with vertex $p_i$;
the other is a quadric cone $Q_R$ over the conic in $\bbP^2$ with vertex the span of $R$.

These two cones intersect in the set of $d-2$ lines 
joining $p_i$ to the points of $R$.
The residual intersection has degree $2(d-1) - (d-2) = d$,
and is the desired rational normal curve. \end{proof}

Applying this iteratively, we have:

\begin{corollary}\label{RNCcorollary}
Given $d+4$ points in general position in $\bbP^d$,
they lie on a rational normal curve of degree $d$
if and only if the following holds.
Fix $d-1$ of the points to form a subset $U$.
For each point $p$ in $U$, we may project from the span of the complement of $p$ in $U$ to a plane.
This gives six points in the plane, namely the projections of the complement of $U$
and the point $p$.
For each such point in $U$, these six points must lie on a conic.
\end{corollary}

Note that in the Corollary, this gives exactly $d-1$ conditions on the points.

\begin{remark}
We see that to apply the above criterion
for $d+4$ points in $\bbP^d$ to lie on a rational normal curve,
we must ultimately make $d-1$ applications of Pascal's Theorem,
which, as we have observed in Section \ref{n=2section},
is one polynomial in the coordinates of the points.

Hence the variety of solutions here is at most codimension $d-1$ in the relevant parameters.
It is easily seen to be exactly $d-1$, by counting dimensions as in the $d=3$ case.

The family of rational normal curves in $\bbP^d$ depends on $d^2+2d-3$ parameters.
(They are all projectively equivalent, and the stabilizer is $3$-dimensional.)
Therefore the $d+4$-tuples of points on such a curve
depend on $d^2+3d+1$ parameters.

The parameters for $d+4$ points in general are $d^2+4d$;
the difference is $d-1$ as claimed.
\end{remark}

\begin{remark}
The criteria presented by Caminata and Schaffler in \cite[Theorem A]{CS}
basically reduces to asking that however one projects from
a subset of $d-2$ points to a plane, the remaining six points lie on a conic,
and this is then rephrased via Pascal's Theorem to a linear condition
on the derived intersection points.
Their argument is algebraic in nature, using the Grassmann-Cayley algebra;
the proof above seems more geometric.
\end{remark}

\subsection{Equations}
In \cite{CGMS}, the authors, among other things, determine equations
for ($d+4$)-tuples of points to lie 
on a rational normal curve of degree $d$ in $\bbP^d$.
The above considerations also enable us to give similar explicit equations for this problem.  We assume $d\geq 3$ here.

Given $d+4$ points in $\bbP^d$, using projective transformations
we may fix the first $d+2$ of them to be
the coordinate points and the point $[1:1:\cdots:1]$.
The coordinates of the remaining two points
give our parameter space for the problem.
We set these coordinates to be
$[a_0:\cdots:a_d]$ and $[b_0:\cdots:b_d]$.
We restrict ourselves to the open subset 
where the points are in general position.

We now apply Corollary \ref{RNCcorollary},
taking $U$ to be the first $d-1$ of the coordinate points:
\[
U = \{[1:0:\cdots:0],[0:1:0:\cdots:0],\ldots,[0:\cdots:0:1:0:0]\}
\]
Using  homogeneous  coordinates $[z_0:z_1:\cdots:z_d]$,
the projection from the complement of the $i$-point ($0 \leq i \leq d-2$ here)
is exactly the projection to the plane with coordinate $[z_i:z_{d-1}:z_d]$.

The six points of the construction in this plane are then given by
$[1:0:0]$ (the image  of the $i$-th point)
$[0:1:0]$, $[0:0:1]$, and $[1:1:1]$ (the projections of the other coordinate points not in $U$)
and $[a_i:a_{d-1}:a_d]$, $[b_i:b_{d-1}:b_d]$ (the projections of the two final points).
We then take the double Veronese images of these six points
to produce the following six points in $\bbP^5$:
\[
\begin{array}{cccccc}
1 & 0 & 0 & 0 & 0 & 0 \\
0 & 1 & 0 & 0 & 0 & 0 \\
0 & 0 & 1 & 0 & 0 & 0 \\
1 & 1 & 1 & 1 & 1 & 1 \\
a_i^2 & a_{d-1}^2 & a_d^2 & a_i a_{d-1} & a_i a_d & a_{d-1} a_d \\
b_i^2 & b_{d-1}^2 & b_d^2 & b_i b_{d-1} & b_i b_d & b_{d-1} b_d 
\end{array}
\]
We must impose that the determinant of the corresponding $6\times 6$ matrix is zero,
which simplifies to the equation
\[
a_i a_d b_{d-1}b_d - a_{d-1}a_d b_i b_d
-a_i a_{d-1} b_{d-1} b_d + a_{d-1}a_d b_i b_{d-1}
+a_i a_{d-1}b_i b_d -a_i a_db_i b_{d-1} = 0.
\]
These equations, as $i$ ranges from $0$ to $d-2$,
give the required equations describing the family of points lying on a rational normal curve, using this parametrization.
This holds for the points $[\underline{a}]$ and $[\underline{b}]$
satisfying some open conditions that ensure the general position of the $d+4$ points.
They are exactly the equations coming from the collinearity of the three
derived points in the plane, via the assertions of Pascal's Theorem.

We note that one may interpret the vanishing of the above determinant
differently.  It is clear that the determinant is equal to the lower right $3\times 3$ minor,
and the vanishing of this minor says exactly that the three points
\[
(a_i a_{d-1},b_i b_{d-1}), (a_i a_d,b_i b_d), (a_{d-1} a_d, b_{d-1} b_d)
\]
are collinear in the affine plane.

These consideration additionally show that our equations define the variety of points on a rational normal curve as a set-theoretic complete intersection.

\section{Quadrics in $\bbP^3$}
One may imagine that an extension of Pascal's Theorem
to quadrics in $\bbP^3$ would be available,
but there does not seem to be an exact analogue available.
In \cite[Note 32]{Chasles} there are presented several statements
for tetrahedra related to quadric surfaces,
but they do not seem to present the same type of phenomena
as a Pascal-type theorem.

\subsection{Ten points}
Quadrics in $\bbP^3$ form a $9$-dimensional complete linear system,
and there is in general a unique quadric through $9$ general points;
there is no quadric through $10$ general points.
The condition that there is a quadric through $10$ points
is easy to describe in terms of the coordinates of the points $\{p_i\}_{1\leq i\leq 10}$.

If $p_i = [a_{i0}:a_{i1}:a_{i2}:a_{i3}]$,
we re-embed these points into $\bbP^9$
via the double Veronese map,
obtaining $10$ points in $\bbP^9$.
They will be dependent in $\bbP^9$ if and only if there is a quadric
through the ten points in $\bbP^3$.
This gives a straightforward polynomial condition to check;
that condition is the determinant of the $10\times 10$ matrix
whose rows are the quadratic monomials in the $a_{ij}$'s,
and hence has degree $20$.

There does not seem to be an alternative construction,
a la Pascal's Theorem for conics in the plane,
to re-write this polynomial as arriving from a different geometric condition.

There are some elementary statements one can make,
taking into account the geometry of a quadric surface.
An example is the following.

\begin{proposition}
Let there be ten points in general position in $\bbP^3$.
Take the ten points, divide them into two sets of five points.
If the ten points lie on a quadric,
then there are twisted cubic curves in $\bbP^3$
passing through each of the two sets of five points,
that meet each other in five additional points.
The converse holds as well: if two such twisted cubics exist,
the ten points lie on a quadric.
\end{proposition}

The proof is easily seen by considering the quadric as $\bbP^1\times\bbP^1$,
and noticing that there are curves of bidegree $(1,2)$ and $(2,1)$
on the quadric through general five points; they will meet in five points.

One may extend the set of conditions and include
requiring that the quadric contains various low degree curves.
We take up several examples of these in the following subsections.

\subsection{Seven Points and One Line}
It is three conditions for a quadric to contain a line,
and therefore if we fix seven general points and one general line,
there should not be a quadric containing them;
the condition that there is one should be one condition
on the parameters describing the points and the line.

We again have a geometric condition in the spirit of the statement for ten points,
but it is not exactly of Pascal type.

\begin{proposition}
Let there be given a general line and seven general points in $\bbP^3$.
If this collection lies on a quadric surface,
there is a rational quartic curve passing through the points
and meeting the line  in three points.  
The converse is also true.
\end{proposition}

Again the proof is immediate upon considering a curve of bidegree $(1,3)$
on the quadric, which will be the rational quartic.

\subsection{Four Points and Two Lines}
Given four points and two lines in general position in $\bbP^3$,
we again should have one condition that they lie on a quadric.
Call the points $p_0,\ldots,p_3$.
There is a unique line $L_i$ containing $p_i$ and meeting the two given lines,
for $1 \leq i \leq 3$.
There is a quadric through the four given points and the two given lines
if and only if
there is a quadric through $p_0$ and $L_1,L_2,L_3$.
Hence this case reduces to the next case.

\subsection{One Point and Three Lines}
In this case there is an elementary Pascal-type condition to state.

\begin{proposition}
Given a general point and general three lines in $\bbP^3$,
they lie on a quadric surface
if and only if
there is a line through the point
meeting each of the three lines.
Equivalently, this happens if the projection from the point
maps the three lines to three concurrent lines in the plane.
\end{proposition}

Again, this is elementary: the three lines (if on a quadric) come from one of the two rulings, and the line through the point is the unique line from the other ruling through the point.

In this case we also offer an algebraic analysis.
Let us choose coordinates
so that the given point $P$ is at $[1:0:0:0]$
and the three given lines each contain the other three coordinate points:
$L_1$ contains $[0:1:0:0]$,
$L_2$ contains $[0:0:1:0]$, and
$L_3$ contains $[0:0:0:1]$.
Then each line is determined by an additional general point
$R_i = [a_{i0}:a_{i1}:a_{i2}:a_{i3}]$ on $L_i$.
We seek the conditions on the $a_{ij}$'s so that $P$ and the $L_i$
lie on a quadric.

That the quadric contains the four coordinate points
means that the quadratic equation has no square terms,
and we are left with the remaining six monomial terms.
The conditions are that the three points $R_i$ are zeroes of the polynomial,
and that the tangent plane to the quadric at the coordinate point
passes through $R_i$ as well.

The first three conditions are quadratic in the $a_{ij}$'s,
and the second three are linear.
The corresponding matrix condition can be easily seen to be that the determinant
of the following $6\times 6$ matrix must be zero:
\[
\begin{pmatrix}
a_{10}a_{11} & a_{10}a_{12} & a_{10}a_{13} & a_{11}a_{12} & a_{11}a_{13} & a_{12}a_{13} \\
a_{20}a_{21} & a_{20}a_{22} & a_{20}a_{23} & a_{21}a_{22} & a_{21}a_{23} & a_{22}a_{23} \\
a_{30}a_{31} & a_{30}a_{32} & a_{30}a_{33} & a_{31}a_{32} & a_{31}a_{33} & a_{32}a_{33} \\
a_{10} & 0 & 0 & a_{12} & a_{13} & 0 \\
0 & a_{20} & 0 & a_{21} & 0 & a_{23} \\
0 & 0 & a_{30} & 0 & a_{31} & a_{32}
\end{pmatrix}
\]
The determinant of this matrix has degree $9$, tridegree $3$ in each of the three sets of variables for the three points $R_i$.
This determinant is:
\[
(a_{21}a_{30} - a_{20}a_{31}) (a_{12}a_{30} - a_{10}a_{32})(-a_{13}a_{20} + a_{10}a_{23})(a_{12}a_{23}a_{31} - a_{13}a_{21}a_{32}).
\]
The first three terms express degenerate cases in which two of the lines meet, and are therefore coplanar.  This geometric condition obviously guarantees the existence of the quadric as the plane containing the two lines that meet, together with the plane through the third line and $P$.  Factoring these out, we obtain the relatively simple algebraic condition that
\[
a_{12}a_{23}a_{31} - a_{13}a_{21}a_{32} = 0
\]
and this expresses exactly the condition that the projection of the three lines from $P$ to the plane $z_0 = 0$ are concurrent.

\subsection{One Conic and Five Points}
In this case we again have a geometric condition.
\begin{proposition}
Let there be given a general conic and five general points in $\bbP^3$.
Then they lie on a quadric surface
if and only if
there is a twisted cubic
passing through the five points
and meeting the conic  at three points.  
\end{proposition}

The conic is of bidegree $(1,1)$ on the quadric,
and the required twisted cubic is of bidegree $(1,2)$.
We note that the twisted cubic is not unique: there is a second one,
having bidegree $(2,1)$.

\subsection{Twisted Cubic and Three Points}
In this case there is a Pascal-type proposition to state, which again follows from relatively elementary considerations.

\begin{proposition}
Given a twisted cubic and three general points in $\bbP^3$,
they lie on a quadric surface
if and only if
the twisted cubic meets the plane in three additional points
which, together with the three given points,
lie on a conic in that plane.
\end{proposition}

In this case the twisted cubic is of bidegree $(1,2)$,
the conic has bidegree $(1,1)$, and the statement follows directly
from the geometry of $\bbP^1\times\bbP^1$.

\section{The Richmond-Segre-Brown Theorem for lines in $\bbP^4$}
There is a remarkable extension of Pascal's Theorem to quadrics in $\bbP^4$,
which is the following construction.

For a quadric in $\bbP^4$ to contain a line is $3$ conditions.
The space of quadrics in $\bbP^4$ is of dimension $14$.
Therefore it is expected that there are no quadrics containing five general lines,
and this is indeed the case.

The Grassmann  variety  of lines in $\bbP^4$ is $6$-dimensional,
so the number of parameters for quintuples of lines is $30$.
Let $\calI$ be the incidence correspondence
consisting of pairs $(Q,C)$ where $Q$ is a quadric, $C$ is a collection of five skew lines, and $C \subset Q$.
The family of lines on a quadric threefold has dimension three,
so that the fiber of $\calI$ to the projective space of quadrics in $\bbP^4$
has dimension $15$.
Hence $\calI$ has dimension $29$.

Now given five general lines on a general quadric,
we claim that there is only that one quadric containing them.
The reason is that the quadrics containing them is a linear system,
and if it is positive dimensional,
the intersection of the given quadric
with a general other member of the system
would be a Del Pezzo surface of degree $4$ in $\bbP^4$
containing the five lines.
However on any such Del Pezzo, there is a line meeting all five skew lines.
This is not possible by the generality of the five lines on the quadric.

Hence the second projection from $\calI$ to the space of five general lines is birational, and so the image has dimension $29$ as well, hence  it has  codimension one.

Therefore it is one condition for a quintuple of lines to lie on a quadric.

What is that condition?  This is the content of a remarkable theorem first addressed by Richmond at the end of the 19th century \cite{R}.  It was rediscovered by B. Segre in a paper in 1945 who does not reference Richmond; he may not have been aware.  In the next year Brown \cite{B} gave a different proof, and Brown references both Richmond and Segre.

It is convenient to make a definition:

\begin{definition}
Fix an ordered set of five lines $L_i$, $1\leq i \leq 5$.
For each $i$, consider the point $R_i$ of intersection of $L_i$
with the span of $L_{i-1}$ and $L_{i+1}$
(indices taken modulo five).
We say that the quintuple of lines is \emph{of RSB type}
(after Richmond, Segre, and Brown)
if those five points $\{R_i\}_{1\leq i\leq 5}$ are dependent in $\bbP^4$.
\end{definition}

The theorem is the following.

\begin{theorem}
Five lines in general position in $\bbP^4$
lie on a quadric
if and only if they are of RSB type.
\end{theorem}

The proofs given in the last centuries by Richmond, Segre, and Brown
are a tour-de-force of four-dimensional synthetic geometry,
involving rather complicated constructions.

Nowadays a computer algebra proof is available,
once we formulate the problem in explicit linear algebra terms.
We will explain now the approach.

We choose two points $P_i$ and $Q_i$ on the line $L_i$.
Using projective transformations we may assume that the $P_i$
are the five coordinates points of $\bbP^4$.

We describe the conditions on the quadric to contain the five lines as follows.
We write the general quadric as the zeroes of a  homogeneous quadratic polynomial   in the
homogeneous coordinate $[z_0:\cdots:z_4]$.
That the quadric contains the coordinate points $P_i$
is equivalent to having no pure square terms in the quadratic polynomial.
This leaves $10$ terms with the monomials $z_iz_j$, $i \neq j$.

Now imposing that the quadric contain the $Q_i$'s
give five linear conditions on the coefficients.
These linear conditions are quadratic in the coordinates of the $Q_i$'s.

Finally imposing that the quadric contains the line $L_i$,
is now equivalent to imposing that the tangent hyperplane to the quadric at $P_i$ contains $Q_i$.
This is five more linear conditions on the coefficients, which are each linear in the coordinates of the $Q_i$'s.

Hence the existence of the quadric amounts to the vanishing of the
determinant $F$ of the $10\times 10$ matrix described by the above ten conditions on the ten coeffients.  This matrix has five rows with entries of degree two in the coordinates and five rows with linear entries in the coordinates. Hence the degree of the polynomial $F$ in these coordinates is $15$.

Let us now describe the condition for the five lines to be of RSB type.
The calculation of the points $R_i$ can be made explicitly as follows.
The hyperplane spanned by $L_{i-1}$ and $L_{i+1}$
is exactly the hyperplane spanned by the four points $P_{i-1}$, $Q_{i-1}$, $P_{i+1}$, $Q_{i+1}$.
Hence the coefficients of this plane are given the $4\times 4$ minors
of the $4\times 5$ matrix whose rows are the coordinates of those four points.
Each of these minors is a quadratic polynomial in the coordinates of the $Q_i$'s.

The intersection of this hyperplane with the line $L_i$
is obtained by taking the parametric equation for $L_i$
and solving for the intersection point;
this gives a point $R_i$ whose coordinate are now cubic polynomials in the
coordinates of the $Q_i$'s.

Now the condition that the $R_i$'s are dependent in $\bbP^4$
is the determinant $G$ of the corresponding $5\times 5$ matrix
whose rows are the coordinates of the $R_i$'s.
Since the $R_i$ are cubic, we see that $G$ is also of degree $15$.

In the same spirit as was discussed in section \ref{n=2section} for Pascal's Theorem in the plane, these two polynomials can be computed explicitly with any modern symbolic algbra package.  We used Sagemath, and verified that $F=G$ as polynomials of degree $15$ in the variable coordinates of the $Q_i$'s.  This proves the theorem.

\begin{remark}
In Segre's paper, he also considered the two relevant determinants above.
He was unable to prove algebraically that these two determinant were equal.
However he argued via synthetic methods that their ratio must be a constant
(and he computed the constant!).  Hence our proof follows his overall line of argument, but the use of the computer simplifies things greatly.
\end{remark}


\begin{thebibliography}{}

\bibitem[B]{B} L.M. Brown: \emph{The configuration determined by five generators of a quadric threefold},  Proceedings of the Edinburgh Mathematical Society , Volume 7 , Issue 4 , October 1946 , 183 - 195.
\bibitem[CGMS]{CGMS}
A. Caminata, N. Giansiracusa, H-B. Moon, L. Schaffler:
\emph{Equations for point configurations to lie on a rational normal curve},
Advances in Mathematics, Volume 340, 15 December 2018,  653-683.
\bibitem[CS]{CS} A. Caminata, L. Schaffler:
\emph{A Pascal's theorem for rational normal curves},
Bull. London Math. Soc., Volume 53 (2021), 1470-1485.
\bibitem[Chasles]{Chasles}
M Chasles: \emph{Apercu historique de l'origine et le developpement des methodes en geometrie particulierement de celles qui se rapportent a la geometrie moderne}, Note 33.
Paris, Gauthier-Villars et fils, 1889.
\bibitem[R]{R} Richmond: \emph{On the condition that five straight lines situated in a space of four dimensions
should lie on a quadric}, Proc. Camb. Phil. Soc., Volume 10, (1900), 210-212. 
\bibitem[S]{S} B. Segre:
\emph{A four-dimensional analogue of Pascal's theorem for conics},
Amer. Math. Monthly, Volume 52, (1945), 119-131.
\bibitem[SM]{SM} N. Stefanovi'c  and M. Milo'sevi'c: 
\emph{A very simple proof of Pascal's hexagon theorem and some applications}, Proc. Indian Acad. Sci. (Math. Sci.), Volume 120, No. 5, November 2010, 619-629.
\end{thebibliography}
\end{document}